\documentclass[11pt]{article}
\usepackage{amsthm}
\usepackage{thmtools}
\usepackage{hyperref}
\usepackage{turnstile}
\usepackage{amsmath}
\usepackage{amsxtra}
\usepackage{amssymb}
\usepackage{lineno}
\declaretheorem[name=Theorem,style=plain,numbered=no]{tm*}

\newcommand{\eqdef}{=_{\mathrm{def}}}
\newcommand{\GCH}{\mathsf{GCH}}
\newcommand{\HOD}{\mathrm{HOD}}
\newcommand{\OD}{\mathrm{OD}}
\newcommand{\HC}{\mathrm{HC}}
\newcommand{\AD}{\mathsf{AD}}
\newcommand{\ZF}{\mathsf{ZF}}
\newcommand{\om}{\omega}
\newcommand{\RR}{\mathbb{R}}
\newcommand{\Ff}{\mathcal{F}}
\newcommand{\Tt}{\mathcal{T}}
\newcommand{\inter}{\cap}
\newcommand{\sub}{\subseteq}
\newcommand{\Coll}{\mathrm{Coll}}
\newcommand{\lh}{\mathrm{lh}}
\newcommand{\Uu}{\mathcal{U}}
\newcommand{\OR}{\mathrm{OR}}
\newcommand{\tu}{\textup}
\newcommand{\Ww}{\mathcal{W}}
\newcommand{\sats}{\models}
\newcommand{\rest}{\!\upharpoonright\!}
\newcommand{\nth}{{\textrm{th}}}
\newcommand{\PP}{\mathbb P}
\newcommand{\BB}{\mathbb B}
\newcommand{\sforces}[2]{\dststile{#2}{#1}}
\newcommand{\es}{\mathbb{E}}
\newcommand{\cross}{\times}
\begin{document}
\title{A long pseudo-comparison of premice in $L[x]$}
\author{Farmer Schlutzenberg\\ farmer.schlutzenberg@gmail.com}
\maketitle
\begin{abstract}We describe an obstacle to the analysis of $\HOD^{L[x]}$ as a core 
model: Assuming 
sufficient large cardinals, for a Turing cone of reals $x$ there are premice $M,N$ in 
$\HC^{L[x]}$ such that the pseudo-comparison of $L[M]$ with $L[N]$ succeeds, is computed in $L[x]$, 
and 
lasts through $\om_1^{L[x]}$ stages.
Moreover, we can take $M=M_1|(\delta^+)^{M_1}$ where $M_1$ is the minimal iterable proper class 
inner model with a Woodin cardinal, and $\delta$ is that Woodin. We can take $N$ such that $L[N]$ 
is $M_1$-like and short-tree-iterable.\end{abstract}
\section{Introduction}\label{sec:Introduction}

A central program in descriptive inner model theory is the analysis of $\HOD^W$, 
for 
transitive models $W$ satisfying $\ZF+\AD^+$; see 
\cite{outline}, \cite{hod_lr_below_theta}, \cite{hod_as_core_model}, \cite{sargsyan}.
For the models $W$ for which it has been successful, the analysis yields a wealth of information 
regarding $\HOD^W$ (including that it is fine structural and satisfies $\GCH$), and in turn about 
$W$.

Assume that there are $\om$ many Woodin cardinals with a measurable above. A primary example of the 
previous paragraph is the analysis of $\HOD^{L(\RR)}$. Work of Steel and Woodin showed that 
$\HOD^{L(\RR)}$ is an iterate of $M_\om$ augmented with a 
fragment of its iteration strategy (where $M_n$ is the minimal iterable proper class inner model 
with $n$ 
Woodin cardinals). The addition of the iteration strategy does not add reals, and so the 
$\OD^{L(\RR)}$ reals are just $\RR\inter M_\om$. The latter has an analogue for $L[x]$, which has 
been known for some time: for a cone of reals $x$, the 
$\OD^{L[x]}$ reals are just $\RR\inter M_1$.
Given this, and further analogies between $L(\RR)$ and $L[x]$ and their respective $\HOD$s, it is 
natural to ask whether there the full $\HOD^{L[x]}$ is an iterate of $M_1$, adjoined with a 
fragment 
of its iteration strategy.
Woodin has conjectured that this is so for a cone of reals $x$; for a precise statement see 
\cite[8.23]{lcfd}. Woodin has proved approximations to this conjecture. He analyzed 
$\HOD^{L[x,G]}$, 
for a cone of reals $x$, and $G\sub\Coll(\om,<\kappa)$ a generic filter over $L[x]$, where $\kappa$ 
is the least inaccessible of $L[x]$; see \cite[8.21]{lcfd} and \cite{hod_as_core_model}. However, 
the conjecture regarding $\HOD^{L[x]}$ is still open.

In this note, we describe a significant obstacle to the analysis of $\HOD^{L[x]}$.
\linebreak

Before proceeding, we give a brief summary of some relevant definitions and facts. We assume 
familiarity with the fundamentals of inner model theory; see \cite{outline}, \cite{fsit}.
One does not really need to know the analysis of $\HOD^{L[x,G]}$,
but familiarity does help in terms of motivation; the system $\Ff$ described below relates to 
that analysis. We do rely on some smaller facts from \cite[\S3]{hod_as_core_model}.
Let us give some terminology, and recall some facts from \cite{hod_as_core_model}.
We say that a premouse $N$ is \emph{pre-$M_1$-like}
iff $N$ is proper class, $1$-small, and has a (unique) Woodin cardinal, denoted $\delta^N$. 
(The notion \emph{$M_1$-like}
of \cite{hod_as_core_model} is stronger; it has some iterability built in.)
Let $P,Q$ be pre-$M_1$-like. Given a normal iteration tree $\Tt$ on $P$, $\Tt$ is 
\emph{maximal} iff 
$\lh(\Tt)$ is a limit and $L[M(\Tt)]$ has no Q-structure for $M(\Tt)$ (so $L[M(\Tt)]$ is 
pre-$M_1$-like 
with Woodin $\delta(\Tt)$). A premouse $R$ is 
a \emph{\tu{(}non-dropping\tu{)} pseudo-normal iterate} of $P$ iff there is a normal tree $\Tt$ on 
$P$ such 
that either $\Tt$ has successor length and $R=M^\Tt_\infty$, the last model of $\Tt$ (and 
$[0,\infty]_\Tt$ does not drop), 
or $\Tt$ is maximal and $R=L[M(\Tt)]$. A
\emph{pseudo-comparison} of $(P,Q)$ is a pair $(\Tt,\Uu)$ of normal iteration trees 
formed according to the usual rules of comparison, such that either $(\Tt,\Uu)$ is a 
successful comparison, or either $\Tt$ or $\Uu$ is maximal. A 
\emph{\tu{(}$z$-\tu{)}pseudo-genericity 
iteration} is defined similarly, formed according to the rules for genericity iterations 
making a real ($z$) generic for Woodin's extender algebra.
We say that $P$ is \emph{normally short-tree-iterable} iff for every 
normal, non-maximal iteration tree $\Tt$ 
on $P$ of limit length, there is a $\Tt$-cofinal wellfounded branch through $\Tt$,
and every putative normal tree $\Tt$ on $P$ of length $\alpha+2$ has wellfounded last model
(that is, we never encounter an illfounded model at a successor stage).
If $P|\delta^P\in\HC^{L[x]}$, then normal short-tree-iterability is absolute 
between $L[x]$ and $V$.
If $P,Q$ are normally short-tree-iterable then there is a 
pseudo-comparison $(\Tt,\Uu)$ of 
$(P,Q)$, and if $\Tt$ has a last model then $[0,\infty]_\Tt$ does not drop, and likewise for $\Uu$.
\linebreak

It has been suggested\footnote{For example, at the AIM Workshop on 
Descriptive inner model theory, June, 2014.} that one might analyze $\HOD^{L[x]}$ using an 
$\OD^{L[x]}$ directed system $\Ff$ such that:
\begin{enumerate}
 \item[--] the nodes of $\Ff$ are pairs $(N,s)$ such that $s\in\OR^{<\om}$ and $N$ is a normally 
short-tree iterable, pre-$M_1$-like premouse with $N|\delta^N\in\HC^{L[x]}$ and such that there is 
an 
$L[N]$-generic filter $G$ for $\Coll(\om,\delta^N)$ in $L[x]$,\footnote{The point of $G$ is 
that we can then use Neeman's genericity iterations, working inside $L[x]$. We \emph{cannot} use 
Woodin's, as closure under Woodin's would produce premice with Woodin cardinal $\om_1^{L[x]}$.}
\item[--] for $(P,t),(Q,u)\in\Ff$, we have $(P,t)\leq_\Ff(Q,u)$ iff $t\sub u$ and $Q$ is a 
pseudo-iterate of 
$P$, and
\item[--] $(M_1,\emptyset)\in\Ff$.
\end{enumerate}
There are also further conditions, regarding the sets 
$s$, strengthening the iterability requirements; these and other details regarding how the 
direct limit is formed from $\Ff$ are not relevant here.

The main difficulty in analyzing $\HOD^{L[x]}$ in this manner is 
in arranging that $\Ff$ be directed.
For this, it seems most obvious to try to arrange that $\Ff$ be closed 
under pseudo-comparison of pairs.

However, we show here that, given sufficient large cardinals, there is a cone of reals $x$ 
such that if $\Ff$ is as above, then $\Ff$ is \emph{not} closed under pseudo-comparison. The proof 
proceeds by finding a node $(N,\emptyset)\in\Ff$ such that, letting $(\Tt,\Uu)$ be the 
pseudo-comparison of $(M_1,N)$, then $\Tt,\Uu$ are in fact pseudo-genericity 
iterations of $M_1,N$ respectively, making reals $y,z$ generic, where 
$\om_1^{L[y]}=\om_1^{L[z]}=\om_1^{L[x]}$. Letting $W$ be the output of the 
pseudo-comparison, we have $W|\delta^W\in L[x]$, so $\om_1^{W[z]}=\om_1^{L[x]}$, which implies that 
$\delta^W=\om_1^{L[x]}$, so $(W,\emptyset)\notin\Ff$. We now proceed to the details.

\section{The comparison}\label{sec:The_Comparison}

For a formula $\varphi$ in the language of set theory (LST),
$\zeta\in\OR$, and $x\in\RR$, let $A_{\varphi,\zeta}^x$ be the set of all
$M\in\HC^{L[x]}$ such that $L[x]\sats\varphi(\zeta,M)$, and $L[M]$ is a normally 
short-tree-iterable pre-$M_1$-like premouse with $\delta^{L[M]}=\OR^M$ and $M=L[M]|\delta^M$.

Note that $\varphi$ does not use $x$ as a parameter. So by absoluteness of normal 
short-tree-iterability (between $L[x]$ and $V$, for elements of $\HC^{L[x]}$), 
$A^x_{\varphi,\zeta}$ is $\OD^{L[x]}$.
So
$A^x_{\varphi,\zeta}$ is a 
collection of premice like those involved in the system $\Ff$ 
(restricted to their Woodins).

\begin{tm*}\label{theorem}
Assume Turing determinacy and that $M_1^\#$ exists and is fully iterable. Then for a cone of 
reals $x$, for every formula $\varphi$ in the LST and every $\zeta\in\OR$, if 
$M_1|\delta^{M_1}\in 
A_{\varphi,\zeta}^x$ then there is 
$R\in A_{\varphi,\zeta}^x$ such that the pseudo-comparison of $M_1$ with $L[R]$ has 
length 
$\om_1^{L[x]}$.
\end{tm*}

\begin{proof} Suppose not. Then we may fix $\varphi$ such that for a cone of $x$, the theorem 
fails for $\varphi,x$. Fix $z$ in this cone with $z\geq_T M_1^\#$. Let $\Ww$ be the $z$-genericity 
iteration on $M_1$ (making $z$ generic for the extender algebra), and $Q=M^\Ww_\infty$.
By standard arguments (see \cite{hod_as_core_model}), $Q[z]=L[z]$,
\[ \lh(\Ww)=\om_1^{L[z]}+1=\delta^Q+1,\]
$Q|\delta^Q=M(\Ww\rest\delta^Q)$, and 
$\Tt\eqdef\Ww\rest\delta^Q$ is the $z$-pseudo-genericity iteration, and $\Tt\in L[z]$.

Let $\BB$ be the extender 
algebra of $Q$ and let $\PP$ be the finite support $\om$-fold product of $\BB$.
For $p\in\PP$ let $p_i$ be the $i^\nth$ component of $p$.
Let $G\sub\PP$ be $Q$-generic, with $z_0=z$ where 
$x\eqdef\left<z_i\right>_{i<\om}$ is the generic sequence of reals. Then
\[ Q[G]=Q[x]=L[x]\]
and 
$x>_Tz$. Let $\zeta\in\OR$ witness the failure of the theorem with respect to $\varphi,x$.
So $M_1|\delta^{M_1}\in A^x_{\varphi,\zeta}$.

By \cite[Lemma 3.4]{Farah} (essentially 
due to Hjorth), $\PP$ is $\delta^Q$-cc in $Q$, so
$\delta^Q\geq\om_1^{L[x]}$, but $\delta^Q=\om_1^{L[z]}$, so $\delta^Q=\om_1^{L[x]}$. So it suffices 
to see that there 
is some $R\in A^x_{\varphi,\zeta}$ such that the pseudo-comparison of $M_1$ with $L[R]$ 
has length $\delta^Q$.

For $e\in\om$ and $y\in\RR$ let $\Phi^y_e:\om\to\om$ be the partial function coded by the 
$e^\nth$ Turing program using the oracle $y$. Let $e\in\om$ be such that $\Phi^z_e$ is 
total and codes $M_1|\delta^{M_1}$. Let $\dot{x}$ be the $\PP$-name for the $\PP$-generic sequence 
of 
reals, and for $n<\om$ let $\dot{z}_n$ be the $\PP$-name for the $n^\nth$ real. Let $p\in 
G$ be such that $p\sforces{Q}{\PP}\psi(\dot{z}_0)$, where $\psi(v)$ asserts ``$\Phi^{v}_e$ is 
total and codes a premouse $R$ such that $R\in A^{\dot{x}}_{\check{\varphi},\check{\zeta}}$, and 
the $v$-pseudo-genericity iteration of $L[R]$ produces a maximal tree $\Uu$ of length 
$\check{\delta^Q}$
with $M(\Uu)=L[\check{\es}]|\check{\delta^Q}$''.
In the notation of this formula,
\[ p\sforces{Q}{\PP}\text{``}R\notin\check{V}\text{''}\text{, because }
p\sforces{Q}{\PP}\text{``}E^\Uu_0\notin M(\Uu)\text{''}.\]
By genericity, we may fix $q\in G$ such that $q\leq p$ and for some $m>0$, 
$q_m=q_0$. Note that 
$q\sforces{Q}{\PP}\psi(\dot{z}_m)$.

Let $\dot{R}_i$ be the $\PP$-name for the premouse coded by $\Phi^{\dot{z}_i}_e$ (or for 
$\emptyset$ if this does not code a premouse). Also let $\dot{z}'_0,\dot{z}'_1$ be the 
$\BB\cross\BB$-names for the two $\BB\cross\BB$-generic reals (in order), and let $\dot{R}'_i$ 
be the $\BB\cross\BB$-name for the premouse coded by $\Phi^{\dot{z}'_i}_e$.

We may fix $r\leq q$, 
$r\in G$, such that 
\begin{equation}\label{eqn:r_defn} r\sforces{Q}{\PP}\text{``}\dot{R}_0\neq\dot{R}_m\text{''}. 
\end{equation}
For otherwise there is $r\leq q$, $r\in G$, such that $r\sforces{Q}{\PP}$``$\dot{R}_0=\dot{R}_m$''.
But since
\[ M_1|\delta^{M_1}=\dot{R}_0^G\notin Q,\]
there are $s,t\in\BB$, $s,t\leq r_0$, such that 
\[ (s,t)\sforces{Q}{\BB\cross\BB}\text{``}\dot{R}'_0\neq\dot{R}'_1\text{''}. \]
Therefore there are $u,v\in\BB$, with $u\leq r_0$ and $v\leq r_m$, such that
\[ (u,v)\sforces{Q}{\BB\cross\BB}\text{``}\dot{R}'_0\neq\dot{R}'_1\text{''}. \]
Let $w\leq r$ be the condition with $w_i=r_i$ for $i\neq 0,m$,
and $w_0=u$ and $w_m=v$. Then
\[ w\sforces{Q}{\PP}\text{``}\dot{R}_0\neq\dot{R}_m\text{''},\]
a contradiction.

So letting $R=\dot{R}_m^G$, we have $R\neq M_1|\delta^{M_1}$ and $R\in A^x_{\varphi,\zeta}$
and $Q|\delta^Q=M(\Uu)$, where $\Uu$ is the $z_m^G$-pseudo-genericity iteration of 
$L[R]$, and $\lh(\Uu)=\delta^Q$. We defined $\Tt$ earlier.
Let $\Tt^*,\Uu^*$ be the padded trees equivalent to $\Tt,\Uu$, such that for each $\alpha$, either 
$E^{\Tt^*}_\alpha\neq\emptyset$ or $E^{\Uu^*}_\alpha\neq\emptyset$, and if 
$E^{\Tt^*}_\alpha\neq\emptyset\neq E^{\Uu^*}_\alpha$ then 
$\lh(E^{\Tt^*}_\alpha)=\lh(E^{\Uu^*}_\alpha)$.
Let $(\Tt',\Uu')$ be the pseudo-comparison of $(M_1,L[R])$.

We claim that 
$(\Tt',\Uu')=(\Tt^*,\Uu^*)$; this completes the proof. For this, we prove by induction on $\alpha$ 
that
\[ (\Tt',\Uu')\rest(\alpha+1)=(\Tt^*,\Uu^*)\rest(\alpha+1).\]
 This is immediate if $\alpha$ is a 
limit, so suppose it holds for $\alpha=\beta$; we prove it for $\alpha=\beta+1$. Let 
$\lambda=\lh(E^{\Tt^*}_\beta)$ or $\lambda=\lh(E^{\Uu^*}_\beta)$, whichever is defined. Because 
$M(\Tt^*)=Q|\delta^Q=M(\Uu^*)$,
the least disagreement between $M^{\Tt^*}_\beta$ and $M^{\Uu^*}_\beta$ has index $\geq\lambda$, so 
we just need to see that $E^{\Tt^*}_\beta\neq E^{\Uu^*}_\beta$.

So suppose that 
$E^{\Tt^*}_\beta=E^{\Uu^*}_\beta$. In particular, both are non-empty.
Then there is $s\in G$ such that $s\leq r$ (see line (\ref{eqn:r_defn})) and 
$s\sforces{Q}{\PP}\sigma$ where $\sigma$ asserts ``For 
$i=0,m$, let $\Tt_i$ be the $\dot{z}_i$-pseudo-genericity iteration of $L[\dot{R}_i]$. Then $\Tt_0$ 
and 
$\Tt_m$ use identical non-empty extenders $E$ of index $\check{\lambda}$.'' Because 
\[ s\sforces{Q}{\PP}\psi(\dot{z}_0)\ \&\ \psi(\dot{z}_m),\]
also $s\sforces{Q}{\PP}\sigma'$, where $\sigma'$ asserts ``Letting 
$E$ be as above, $E\sub L[\check{\es}]|\check{\lambda}$, but $E\notin\check{V}$''; here
$E^{\Tt^*}_\beta\notin Q$ because $\lambda$ is a 
cardinal of $Q$. 
But since $\Tt_i^G$ is computed in $Q[z_i^G]$ (for $i=0,m$)
we can argue as before (as in the proof of the existence of $r$ as in line (\ref{eqn:r_defn})) to 
reach a contradiction.
\end{proof}

A slightly simpler argument, using $\BB\cross\BB$ instead of $\PP$,
proves the weakening of the theorem given by dropping the parameter $\zeta$.

\bibliographystyle{plain}
\bibliography{a_long_comparison}
\end{document}